\documentclass[12pt,a4paper]{amsart}
\usepackage{amsfonts}
\usepackage{amsthm}
\usepackage{amsmath}
\usepackage{amscd}
\usepackage[latin2]{inputenc}
\usepackage{t1enc}
\usepackage[mathscr]{eucal}
\usepackage{indentfirst}
\usepackage{graphicx}
\usepackage{graphics}
\usepackage{pict2e}
\usepackage{epic}
\numberwithin{equation}{section}
\usepackage[margin=2.9cm]{geometry}
\usepackage{epstopdf}

\theoremstyle{plain}
\newtheorem{Th}{Theorem}[section]

\newtheorem{Cor}[Th]{Corollary}
\newtheorem{Prop}[Th]{Proposition}

 \theoremstyle{definition}

\newtheorem{?}[Th]{Problem}

\begin{document}

\title[Spin $m$-quasi-Einstein manifolds \& the Hitchin-Thorpe Inequality]{Compact 4-Dimensional Spin Gradient $m$-quasi-Einstein Manifolds Satisfy the Hitchin-Thorpe Inequality when $m\ge 1$}

\author{Brian Klatt}

\address{Rutgers University \\ Department of Mathematics \\
New Brunswick, NJ \\ United States} 

\email{brn.kltt@gmail.com}

 \subjclass[2010]{Primary: 53C25}

 \keywords{Ricci soliton, $m$-quasi-Einstein manifold, Hitchin-Thorpe inequality}

\begin{abstract} We prove that a compact, connected, and oriented 4-dimensional gradient $m$-quasi-Einstein manifold with $m\in [1, \infty]$ which is additionally a spin manifold must satisfy the Hitchin-Thorpe Inequality. We show further that the homeomorphism-type of the universal cover of such a manifold is either $S^4$ or a connected sum of some number of $S^2\times S^2$ when the potential function is nontrivial.
\end{abstract}

\maketitle

\section{Introduction}

In \cite{Thorpe}, Thorpe defined a natural notion of curvature for a $2k$-plane in the tangent space of a Riemannian manifold which generalized the usual notion of curvature for a $2$-plane. He then proved that if the curvature of a $2k$-plane and its complementary plane are identical in a compact oriented $4k$-manifold then there holds an inequality between the manifold's Euler characteristic and $k^{th}$ Pontryagin number, namely, $(2k)!\,\chi(M^{4k})\geq (k!)^2\,|p_k(M^{4k})|$. 

It was not noted by Thorpe at the time, but when $k=1$ the hypothesis on the curvature is identical to the condition that $(M^4, g)$ is an Einstein manifold; that is, the Ricci tensor is in constant proportion to the metric tensor. In this case the inequality of Thorpe reduces to $2\chi\geq |p_1|$. Hitchin later discovered this in \cite{Hitchin} and characterized the topology when equality occurs. The inequality $2\chi\geq 3|\tau|$ or $2\chi\pm 3\tau\geq 0$ on a compact oriented $4$-dimensional Einstein manifold came to be known as the Hitchin-Thorpe inequality; here $\tau$ is the signature of the $4$-manifold, and the conversion between $p_1$ and $\tau$ is a consequence of the Hirzebruch Signature Theorem.

In \cite{Cao}, H.-D. Cao asked whether the same inequality might hold on a compact oriented $4$-dimensional Ricci soliton. By definition this is a Riemannian manifold satisfying
$$r+\frac{1}{2}\pounds_X g=\lambda g$$
where $\lambda$ is a constant. Ricci solitons are generalizations of Einstein manifolds that have an important role in the singularity analysis of the Ricci flow, as they are special solutions to the flow, evolving only by diffeomorphisms and homothetic scaling (see \cite{Cao}). Therefore the Hitchin-Thorpe inequality would have implications for an analysis of singularities of $4$-dimensional Ricci flow. Some work has been completed on this question with additional analytic or geometric hypotheses; see for example \cite{Fer}, \cite{Ma}, \cite{Tadano}. 

In recent years, the notion of an $m$-quasi-Einstein manifold has received a fair bit of attention. By definition this is a Riemannian manifold such that \cite{WylKen}
$$r+\frac{1}{2}\pounds_X g-\frac{1}{m}X^{\flat}\otimes X^{\flat}=\lambda g$$ where $\lambda$ is a real constant and $m$ is an extended real number. By convention, when $m=\pm\infty$, we obtain a Ricci soliton, and when $m=0$ we have an Einstein manifold. The main interest in such spaces has been when $X=\nabla f$, due to a connection with Einstein warped products and conformally Einstein manifolds, among other things; see \cite{Case}. These we call \emph{gradient} $m$-quasi-Einstein manifolds.

In this work, we prove, with an additional topological hypothesis, the conjecture of Cao in the wider context of gradient $m$-quasi-Einstein manifolds where the range of $m$ is appropriately restricted. Precisely, we prove the following

\begin{Th} \label{main} Let $(M^4, g)$ be a compact, connected, oriented \& spin $4$-dimensional gradient $m$-quasi-Einstein manifold with $m\in [1,\infty]$. Then $M$ satisfies the Hitchin-Thorpe inequality, i.e. $2\chi\pm 3\tau\ge0$, and if $(M^4, g)$ is a nontrivial such gradient $m$-quasi-Einstein manifold then we have the strict inequality $2\chi\pm 3\tau>0$.
\end{Th}

By pushing our analysis only slightly further we can characterize the topology more completely when our hypothesis holds.

\begin{Th} \label{aux} If $(M^4, g)$ is a compact, connected, oriented \& spin $4$-dimensional nontrivial gradient $m$-quasi-Einstein manifold with $m\in [1,\infty]$, then the universal cover $\tilde{M}^4$ satisfies $\tilde{M}^4\approx S^4\#k(S^2\times S^2)$ for some $k$.
\end{Th}

\section{Preliminaries}

\subsection{Conventions and Notation}

We will assume that all manifolds are connected. The lower case letters $r$ and $s$ will denote the Ricci and scalar curvatures of a Riemannian metric, respectively. The Laplacian of a smooth function on a Riemannian manifold is the trace of its Hessian, $\Delta f=g^{ij}\nabla_i\nabla_j f$, where $\nabla$ is the Levi-Civita connection. The divergence of a vector field $X=X^i$ and a one-form $\omega=\omega_i$ are given by $\mathrm{div}X=\nabla_i X^i$ and $\mathrm{div}\,\omega=g^{ij}\nabla_i\omega_j$, while the divergence of a symmetric covariant two-tensor $h$ is the one-tensor defined by $\mathrm{div}\,h=g^{ij}\nabla_i h_{jk}$. Finally, the signature of a smooth oriented $4k$-manifold is the difference between the dimensions of the largest subspaces of $H^{2k}(M^4, \mathbb{R})$ on which the quadratic form $([\phi], [\psi] )\mapsto \int \phi\wedge\psi$ is positive and negative definite.

\subsection{$m$-quasi-Einstein manifolds}

An \emph{$m$-quasi-Einstein manifold} is a Riemannian manifold $(M, g)$ which satisfies 
\begin{equation}
r+\frac{1}{2}\pounds_X g-\frac{1}{m}X^{\flat}\otimes X^{\flat}=\lambda g \label{eq:RS}
\end{equation}
for some vector field $X$, real constant $\lambda$, and extended real number $m$. If $m=\pm\infty$ then $\frac{1}{m}=0$ by convention; this is called a \emph{Ricci soliton}. When $m=0$ the equation reduces, again by convention, to $r=\lambda g$; that is, we have an Einstein manifold. If $X=0$ the $m$-quasi-Einstein manifold is called $\emph{trivial}$, and is consequently an Einstein manifold. If $\lambda>0$, $\lambda<0$, or $\lambda=0$ we will say that $(M, g)$ is \emph{shrinking}, \emph{expanding}, or \emph{steady}, respectively.

We state a couple well-known facts about compact gradient $m$-quasi-Einstein manifolds. First, we have the following (see Proposition 4.22 in \cite{Case})

\begin{Prop}\label{TrivOrShrink}
A compact, gradient, expanding or steady $m$-quasi-Einstein manifold is trivial.
\end{Prop}

Thus essentially only shrinking gradient $m$-quasi-Einstein manifolds are interesting in the compact case. Furthermore, for certain values of $m$ in this case we have the following (see Proposition 3.6 and Remark 3.8 of \cite{CaseEtAl})

\begin{Prop}\label{PosScal}
A compact, gradient, shrinking $m$-quasi-Einstein manifold with $m\in [1,\infty]$ has positive scalar curvature.
\end{Prop}

We will additionally need the following Myers-type result of W. Wylie \cite{Wylie} which has implications for the topology of compact shrinking $m$-quasi-Einstein manifolds when $m>0$:
\begin{Prop}\label{Wylie}
A complete Riemannian manifold satisfying
\[r+\frac{1}{2}\pounds_X g \geq \lambda g\]
for some vector field $X$ and constant $\lambda > 0$
has finite fundamental group.
\end{Prop}

\begin{Cor}\label{RSBetti}
A compact shrinking $m$-quasi-Einstein manifold with $m>0$ has finite fundamental group, and therefore the first Betti number $b_1=0$.
\end{Cor}
\begin{proof}
Since shrinking $m$-quasi-Einstein manifolds with $m>0$ have $\lambda>0$ and $\frac{1}{m}X^{\sharp}\otimes X^{\sharp}\geq 0$, and compact Riemannian manifolds are complete, Wylie's theorem applies. Since the first homology group is a quotient of the fundamental group (by the Hurewicz theorem) it must be a finite group as well. So $b_1=\mathrm{rank}(H_1)=0$.
\end{proof}

\subsection{Spin manifolds}
Our references for the material of this section are \cite{LawMich} and \cite{Fried}. A smooth \emph{spin manifold} is an oriented Riemannian manifold whose principal $\mathrm{SO}(n)$-bundle $P\rightarrow M$ of oriented orthonormal frames lifts equivariantly to a principal $\mathrm{Spin}(n)$-bundle $\tilde{P}\rightarrow M$. The topological obstruction to constructing such a bundle is well-known to be the second Stiefel-Whitney class $w_2$, that is, $w_2=0$ if and only if such a lifting exists. The main utility of having such a structure in geometry is to construct the associated Hermitian vector bundle of spinors which has a canonical Levi-Civita connection $\nabla$, lifted from the connection on the tangent bundle and compatible with the Hermitian inner product. In addition to the connection, there is a canonical self-adjoint operator $D$ on sections of the spinor bundle, called the Dirac operator, and the following well-known equation of Lichnerowicz holds:
\begin{equation}
D^2\psi=\nabla^*\nabla\psi + \frac{s}{4}\psi \label{eq:Lich}
\end{equation}
If $M$ is compact and such that $s\ge 0$ and $s>0$ at a point, then pairing \eqref{eq:Lich} with $\psi$ and integrating immediately yields that $\psi$ is $\nabla$-parallel so must have constant length, and this length must be zero since $s>0$ somewhere. Thus $\psi=0$ identically i.e. $\mathrm{ker}(D)=0$. It is a consequence of the Atiyah-Singer Index Theorem that, if $M$ is $4$-dimensional and $\mathrm{ker}(D)=0$, then the signature $\tau$ of $M$ satisfies $\tau=0$.
\begin{Prop}\label{spintau}
A compact $4$-dimensional Riemannian spin manifold such that $s\ge 0$ and $s>0$ at a point must have $\tau=0$.
\end{Prop}

\subsection{Topology of smooth 4-manifolds}
The work of Freedman and Donaldson in the 1980s revolutionized our understanding of smooth $4$-manifolds, and leads to the following classification result \cite{LeBrun}.
\begin{Prop}\label{4manifolds}
Two smooth compact simply-connected oriented $4$-manifolds are orientedly homeomorphic if and only if
\begin{itemize}
\item they have the same Euler characteristic $\chi$
\item they have the same signature $\tau$; and
\item both are spin, or both are non-spin.
\end{itemize}
\end{Prop}

This proposition has the following corollary:

\begin{Cor}\label{4dList}
Any smooth spin simply-connected oriented $4$-manifold such that $\frac{11}{8}|\tau|\leq b_2$ is orientedly homeomorphic to $S^4\#mK3\#n\overline{K3}\#p(S^2\times S^2)$, $m,\,n,\,p\geq 0$. Any smooth non-spin simply-connected oriented $4$-manifold is orientedly homeomorphic to $k\mathbb{CP}^2\#l\overline{\mathbb{CP}}^2$, $k,\,l\geq 0$ and not both zero.
\end{Cor}

This is established by merely computing the Euler characteristic and signature of the listed manifolds and showing that one can always choose the integers appropriately to match the Euler characteristic and signature of any given $4$-manifold satisfying the hypotheses; then apply Proposition \ref{4manifolds}.

\section{Proof of Theorem \ref{main}}

In this section we present the proofs of Theorem~\ref{main} and Theorem \ref{aux}.

\subsection{Proof of Theorem \ref{main}}
\begin{proof}
Since $M^4$ is a compact gradient $m$-quasi-Einstein manifold, by Prop.~\ref{TrivOrShrink} it is either trivial, and thus an Einstein manifold, or shrinking. If the former, we are done by Hitchin and Thorpe's original result. If the latter, then since $m\in [1,\infty]$ we have $s>0$ by Prop.~\ref{PosScal}. Now because $M^4$ is additionally a spin manifold, we find $\tau=0$ by Prop.~\ref{spintau}. As $M^4$ is a compact shrinking $m$-quasi-Einstein manifold, by Corollary~\ref{RSBetti} we have $b_1=0$ but also $b_3=0$ by the compactness of $M$ and Poincar\'{e} Duality. Therefore $2\chi\pm 3\tau=2(b_0+b_2+b_4)=2(2+b_2)\ge 4>0$.
\end{proof}

\subsection{Proof of Theorem \ref{aux}}
\begin{proof}
Since $\pi_1(M^4)$ is finite by Corollary \ref{RSBetti} the universal cover $\tilde{M}$ is a compact oriented spin nontrivial gradient $m$-quasi-Einstein manifold by pulling everything back via the covering map. Since $\tilde{M}$ is a compact nontrivial gradient $m$-quasi-Einstein manifold, we have by the proof of Theorem \ref{main} that $\frac{11}{8}|\tau(\tilde{M})|=0\leq b_2(\tilde{M})$ so $\tilde{M}\approx S^4\#mK3\#n\overline{K3}\#p(S^2\times S^2)$ by Corollary \ref{4dList}. Now compute using $\tau(S^4)=\tau(S^2\times S^2)=0$, $\tau(K3)=-16$, $\tau(\overline{K3})=16$ that $0=\tau(\tilde{M})=-16(m-n)$; thus $m=n$. However, $m(K3\#\overline{K3})\approx 22(S^2\times S^2)$ by Corollary \ref{4dList} since both are spin manifolds with the same Euler characteristic and signature. Thus $\tilde{M}\approx S^4\#(22m+p)(S^2\times S^2)$, and setting $k=22m+p$ completes the proof.
\end{proof}

\section{Concluding Remarks}

\subsection{The Ricci Soliton Case}

Our results are true in the $m=\infty$ case without the gradient hypothesis because of the well-known result of Perelman (see Remark 3.2 of \cite{Per}) that compact Ricci solitons are gradient Ricci solitons. However there are compact nontrivial nongradient $m$-quasi-Einstein manifolds for finite $m$ \cite{BRS}. With this in mind, if one tries to prove the key equation used in \cite{CaseEtAl} to prove Prop.~\ref{PosScal} without the gradient hypothesis, one is stymied by additional terms that only vanish when $X^{\flat}$ is closed. Stating our results with this hypothesis is, however, no more general than the gradient case $X^{\flat}=\mathrm{d}f$, since by Cor.~\ref{RSBetti} a closed $1$-form is here exact.

An interesting aspect of the Ricci soliton case is that it is not actually known whether nontrivial compact $4$-dimensional spin Ricci solitons exist. So while Theorem \ref{main} rules out the possibility of a counterexample to the Hitchin-Thorpe inequality among such manifolds, it is not clear whether or not this is so merely by virtue of there being no such manifolds whatsoever!

\begin{?} Can one can construct a nontrivial Ricci soliton metric on a compact spin $4$-dimensional manifold?
\end{?}

\noindent Theorem \ref{aux} rules out some of the more complicated topologies. It would be interesting to try to construct such a metric on $S^2\times S^2$ given the existence of an elementary construction of an Einstein-Weyl metric on the same manifold in \cite{Ped}.

\subsection{Expanding Cao's Question}

\noindent Let us consider the original inequality of Hitchin-Thorpe for Einstein manifolds, the restricted results of this paper, and the Hitchin-Thorpe inequality for Einstein-Weyl manifolds originally proved in \cite{PedPoon}. An Einstein-Weyl manifold can be defined merely as an $m$-quasi-Einstein manifold where $m=-(n-2)$ and $\lambda$ is allowed to be a smooth function instead of merely a constant. (One can see this by comparing the $m$-quasi-Einstein equation to the Einstein-Weyl equation on page 100 of \cite{Ped}.) In the language of \cite{BG}, an Einstein-Weyl manifold is a \emph{generalized} $m$-quasi-Einstein manifold. We take these results as evidence in favor of an affirmative answer to the following. 

\begin{?} Does the Hitchin-Thorpe Inequality hold for any compact oriented 4-dimensional generalized $m$-quasi-Einstein manifold?
\end{?}

The author knows of no examples that demonstrate the answer to be `no.' The following is a list of known compact nontrivial generalized $m$-quasi-Einstein manifolds in dimension four; the underlying manifolds all satisfy the Hitchin-Thorpe inequality:

\begin{itemize}
\item According to \cite{BR}, there is a nontrivial generalized gradient $m$-quasi-Einstein metric on $S^4$ for any $m\neq 0, \pm\infty$ (in fact for any $S^n$). These are nontrivial in the technical sense stated above, that is, the gradient field is not the zero field. However the metric is not new; it is the standard round metric.
\item There is a one-parameter family of nontrivial $m$-quasi-Einstein metrics on 
$\mathbb{CP}^2\#\overline{\mathbb{CP}}^2$ for $m\in (1,\infty]$ which limits to the well-known Cao-Koiso soliton (\cite{CaoSoliton}, \cite{Koiso}) when $m=\infty$, as follows from Theorem 5.8 in \cite{Case}; see that work for details and further references. 
\item The Ricci soliton of Wang and Zhu \cite{WangZhu} is on the manifold $\mathbb{CP}^2\#\overline{\mathbb{CP}}^2\#\overline{\mathbb{CP}}^2$.
\item There are Einstein-Weyl metrics which are not conformal to an Einstein metric on $S^2\times S^2$ and $\mathbb{CP}^2\#\overline{\mathbb{CP}}^2$ according to Corollary 4.3 of \cite{Ped}
\end{itemize}

The relaxation of the constant $\lambda$ to a function seems unwarranted based on the arguments of this paper, but our arguments clearly don't go to the heart of the matter anyway, attached as they are to the case of spin manifolds. However, the $m=-(n-2)$ case is established, as noted above, and if we revisit the original arguments of Hitchin and Thorpe in the case of Einstein manifolds, we find they do not require that $\lambda$ is constant. Of course $\lambda$ is constant by the Contracted Bianchi Identity, but this is entirely irrelevant to the line of argument in the proof. We interpret this as evidence that a natural proof of the affirmative answer to our question will not distinguish whether $\lambda$ is constant or a function.

The author is currently pursuing this question.

\section{Acknowledgements} I would like to thank Lehigh University and Rutgers University for their institutional support, and Huai-Dong Cao for his encouragement when I, as his graduate student, began thinking about his question regarding the Hitchin-Thorpe inequality for Ricci solitons. I am also grateful to Homare Tadano for several comments which improved this article.

\end{document}